\documentclass{amsart}

\usepackage{graphicx}
\usepackage{graphics}
\usepackage{amsmath}
\usepackage{amssymb}
\usepackage{amsfonts}
\usepackage{amsxtra}
\usepackage{enumerate}
\usepackage[all]{xy}
\usepackage{amsthm}
\usepackage{amscd} 
\usepackage{dsfont}
\usepackage{newlfont}
\usepackage{longtable}
\usepackage{url}

\newcommand{\pt}{pt}
\newcommand{\deRham}{\Omega}
\newcommand{\deRhama}{\Omega_X}
\newcommand{\Sn}{\mathfrak{S}_n}
\newcommand{\Snull}{\mathfrak{S}_0}
\newcommand{\Seins}{\mathfrak{S}_1}
\newcommand{\An}{\mathfrak{A}_n}
\newcommand{\Azwei}{\mathfrak{A}_2}

\newcommand{\ra}{\rightarrow}

\newcommand{\hra}{\hookrightarrow}
\newcommand{\sira}{\;{\stackrel{\sim}{\rightarrow}}\;}

\newcommand{\tha}{\twoheadrightarrow}

\newcommand{\BZ}{{\mathbb{Z}}}

\newcommand{\BR}{{\mathbb{R}}}
\newcommand{\BC}{{\mathbb{C}}}

\newcommand{\BK}{{\mathbb{K}}}

\newcommand{\ol}[1]{{\overline{#1}}}

\newcommand{\os}{\overset}

\newcommand{\cs}[2]{{#1}_{#2}} 
\newcommand{\Str}{S}

\newcommand{\Strr}{T}

\newcommand{\im}{\text{im}\hspace{0,5mm}}

\newcommand{\Per}[1]{\text{Per}_{{#1}}}

\newcommand{\wXCons}{\operatorname{Sh}_{w, \hspace{0,5mm}c}(X)}
\newcommand{\XCons}{\operatorname{Sh}_c(X)}
\newcommand{\Dbwc}{D^b_{w,\hspace{0,5mm}c}(X)}
\newcommand{\Dpluswc}{D^+_{w,\hspace{0,5mm}c}(X)}
\newcommand{\Dsternwc}{D^*_{w,\hspace{0,5mm}c}(X)}
\newcommand{\Dbc}{D^b_c(X)}

\newcommand{\Dsternc}{D^*_c(X)}
\newcommand{\mc}{\mathcal}

\newcommand{\Hom}{\operatorname{Hom}}
\newcommand{\End}{\operatorname{End}}

\newcommand{\Ext}{\operatorname{Ext}}

\newcommand{\Staralpha}{\text{Et}_\Str}

\newcommand{\Starbeta}{\text{Et}_\Strr}

\renewcommand{\int}{int}

\newcommand{\Rep}{\operatorname{Rep}}
\newcommand{\Repee}{\operatorname{Rep_f}}
\newcommand{\Sh}[1]{\operatorname{Sh}(#1)} 
\newcommand{\sh}[1]{\operatorname{Sh_f}(#1)} 

\newcommand{\Ring}{R}
\newcommand{\field}{\BK}
\newcommand{\exponent}{q}
\newcommand{\exponentzwei}{q}

\newcommand{\Kugelschale}{S^2}
\newcommand{\nKugelschale}{S^n}

\newtheorem{theorem}[equation]{Theorem}

\newtheorem{proposition}[equation]{Proposition}
\newtheorem{corollary}[equation]{Corollary}
\newtheorem{lemma}[equation]{Lemma}

\theoremstyle{definition}

\newtheorem{definition}[equation]{Definition}

\newcommand{\he}[1]{h^{\hspace{-0,1mm}e}_{{#1}}}
\newcommand{\hp}[1]{h^{\hspace{-0,1mm}\footnotesize{p}}_{#1}}

\begin{document}

\title[Formality of the constructible derived category for spheres]{Formality of the constructible derived category for spheres:
A combinatorial and a geometric approach}
\author{Anne Balthasar}
\address{%
\hspace{-0.51cm} Anne Balthasar, London School of Economics,
Department of Mathematics, Houghton St, London WC2A 2AE, E-mail
address: a.v.balthasar@lse.ac.uk} 

\begin{abstract}
We describe the constructible derived category of sheaves on the
$n$-sphere, stratified in a point and its complement, as a dg module
category of a formal dg algebra. We prove formality by exploring two
different methods: As a combinatorial approach, we reformulate the
problem in terms of representations of quivers and prove formality
for the 2-sphere, for coefficients in a principal ideal domain. We
give a suitable generalization of this formality result for the
2-sphere stratified in several points and their complement. As a
geometric approach, we give a description of the underlying dg
algebra in terms of differential forms, which allows us to prove
formality for $n$-spheres, for real or complex coefficients.
\end{abstract}

\maketitle

\section{Introduction}


Let $X$ be a stratified topological space. A sheaf of modules over a
principal ideal domain $\Ring$ on $X$ is called constructible if it
is locally constant along the strata, and the stalks are finitely
generated. Let $D^b(X)$ be the bounded derived category of sheaves
of $\Ring$-modules on $X$. The constructible derived category
$D^b_c(X)$ is the full triangulated subcategory of $D^b(X)$
consisting of bounded complexes of sheaves with constructible
cohomology.

In this paper we give an algebraic description of $D^b_c(X)$ in
terms of dg module categories. We would like to prove for the
special case of the $n$-spheres that the corresponding dg algebra is
formal, i.e. quasi-equivalent to a dg algbra with trivial
differential. For this purpose, we give general combinatorial and
geometric descriptions of $D^b_c(X)$, which we then use to prove the
desired formality results.

The combinatorial approach relies on a generalization of a result of
\cite{KS}, who showed that for a simplicial complex $\mc{K}$, we
have an equivalence of categories
$$
D^b_c(\mc{K}) \simeq D^b(\operatorname{Sh}_c(\mc{K}))
$$
where $\operatorname{Sh}_c(\mc{K})$ denotes the category of
constructible sheaves on $\mc{K}$ (with respect to the natural
stratification). We prove a similar statement for more general
stratifications, which we call acyclic. As a consequence, we get the
following combinatorial description of $D^b_c(X)$:

\begin{theorem}\label{main3}(cf. Theorem \ref{quivers})
To every acyclic stratification on a topological space $X$ of finite
homological dimension we can assign a canonical quiver $Q$ with
relations $\rho$, such that we get a natural equivalence
$$
\Dbc \sira D^b(\Repee(Q,\rho))
$$
where $\Repee(Q,\rho)$ denotes the category of representations of
$(Q,\rho)$ with finitely generated stalks.
\end{theorem}

For the geometric approach, we give a description of $D^b_c(X)$ in
terms of differential forms, for the case that $X$ is a
differentiable manifold, stratified in a point and its complement,
and the base ring is the field $\field$ of real or complex numbers.
Inspired by results of \cite{DGMS}, who proved that the de Rham
algebra $\deRham$ on a K{\"a}hler manifold is formal, we define a
suitable ``extended'' de Rham algebra $\mc{M}(\deRham)$, for which
we get the following characterization of $D^b_c(X)$:

\begin{theorem}\label{main4}(cf. Theorem \ref{deRhamTheorem})
Let $X$ be a second countable differentiable manifold and $\pt$ a
point in $X$ such that $X\backslash\pt$ is simply connected. We have
an equivalence of triangulated categories
$$
\Dbc \simeq \mc{D}^f_{\mc{M}(\deRham)}
$$
where $\mc{D}^f_{\mc{M}(\deRham)}$ is a suitable subcategory of the
derived dg module category $(\mc{D}_{\mc{M}(\deRham)})^\circ$.
\end{theorem}

We use those two descriptions of $D^b_c(X)$ to prove formality for
the $n$-sphere $S^n$, stratified in a point and its complement. For
this, let $\mc{L}_1$ be the skyscraper at the point of the
stratification, and $\mc{L}_2=\cs\Ring{S^n} [\lfloor n/2\rfloor]$
the constant sheaf on the $n$-sphere shifted by $\lfloor
n/2\rfloor$. For $\mc{L} = \mc{L}_1\oplus\mc{L}_2$, consider the dg
algebra $\Ext(\mc{L})$ of self-extensions of $\mc{L}$ in $D^b(S^n)$,
with trivial differential. Let $\mc{D}_{\Ext(\mc{L})}$ be the
derived category of dg modules over $\Ext(\mc{L})$ as defined in
\cite{BL94}, and $\mc{D}^f_{\Ext(\mc{L})}$ the full triangulated
subcategory of $(\mc{D}_{\Ext(\mc{L})})^\circ$ generated by the dg
modules $\mc{L}p_1$ and $\mc{L} p_2$, for $p_i:\mc{L}\ra\mc{L}_i$
the projections. We prove the following formality results:
\begin{theorem}\label{main1} \emph{(cf. Theorems
\ref{formalityfornspheres}, \ref{onepointstratification})}
\begin{enumerate}
\item[(i)] For $\Ring$ any principal ideal domain, we have an equivalence
$$D^b_c(S^2) \simeq \mc{D}^f_{\Ext(\mc{L})}$$
\item[(ii)] If the base ring $\Ring$ is $\BR$ or $\BC$, we have for
$n\geq 2$ an equivalence
$$D^b_c(S^n) \simeq \mc{D}^f_{\Ext(\mc{L})}$$
\end{enumerate}
\end{theorem}

The main step in the proof of this result is to take an injective
resolution $I$ of $\mc{L}$ and prove that $\End I$ is a formal dg
algebra. This argument can be generalized to the following setting:
\begin{theorem}\label{main2} \emph{(cf. Theorem
\ref{npointstratification})} For the 2-sphere $\Kugelschale$
stratified in $m$ points and their complement, let
$\mc{L}_1,\dots,\mc{L}_m$ be the skyscrapers at the points of the
stratification, and $\mc{L}=\cs\Ring{S^2}[1]$ the constant sheaf
shifted by 1. For any injective resolution $I$ of
$\mc{L}_1\oplus\dots\oplus \mc{L}_m\oplus\mc{L}$, the corresponding
dg algebra $\End I$ is formal.
\end{theorem}

The idea of describing $D^b_c(S^n)$ as a dg module category of a
formal dg algebra goes back to an analogous conjecture in the
corresponding equivariant setting. For $X$ a projective variety on
which a complex reductive group $G$ acts with finitely many orbits,
let $D^b_{G,c}(X)$ be the equivariant constructible derived category
as introduced in \cite{BL94}. In \cite{Soergel01} it has been
conjectured implicitely that $D^b_{G,c}(X)$ is equivalent to
$\mc{D}^f_{\mc{E}}$, where $\mc{E}$ is the equivariant extension
algebra on $X$, i.e. a dg algebra with trivial differential. So far,
this conjecture has been proved in special cases
(\cite{Lunts},\cite{Guillermou}); a proof for flag varieties is due
to appear in \cite{Olaf}.

Our above-mentioned results explore in how far the equivariant case
can be translated to the non-equivariant setting. For the proofs of
Theorems \ref{main1}(i) and \ref{main2} we use the combinatorial
description of $D^b_c(X)$ given in Theorem \ref{main3}. This method
works for coefficients in any principal ideal domain. For the proof
of Theorem \ref{main1} (ii) we exploit Theorem \ref{main4}, hence it
only works for coefficients in $\BR$ or $\BC$.

Our paper is organized as follows: Section \ref{Kap2} is a summary
of relevant results on the constructible derived category and dg
module categories. In section \ref{Kap3}, we describe the bounded
constructible derived category $\Dbc$, for $X$ a differentiable
manifold stratified in a point and its complement, using
differential forms. This characterization allows us to prove
formality for $n$-spheres with this specific stratification. In
section \ref{Kap4} we introduce acyclic stratifications and give a
combinatorial description of $D^b_c(X)$ in terms of representations
of quivers. Using this description, we prove in section \ref{Kap5}
formality for the 2-sphere, stratified in several points and their
complement.

This paper is a condensed version of my diploma thesis
\cite{Balthasar}, which I wrote in Freiburg in 2005. I am deeply
indebted to Wolfgang Soergel for his support and many fruitful
discussions, and to Olaf Schn{\"u}rer for lots of helpful comments.

\section{The constructible derived category as a dg module
category}\label{Kap2}

\subsection{Conventions}\label{Conventions}
We fix a commutative unitary Noetherian ring $\Ring$ of finite
homological dimension. By a sheaf on a topological space $X$ we
always mean a sheaf of modules over $\Ring$. We denote the category
of such sheaves by $\Sh X$, and by $\sh X$ the full subcategory of
sheaves with finitely generated stalks. As usual, we write $D^b(X)$
for the bounded derived category of sheaves of $\Ring$-modules on
$X$.

We denote the constant sheaf with stalk $\Ring$ by $\cs\Ring X$. For
a continous map $f: X\ra Y$ we write $f^*$ for the inverse image
functor, $f_*$ for the direct image and $f_!$ for the direct image
with compact support (if it exists). In this paper we are mostly
going to use locally closed inclusions and would like to recall the
following properties (cf. \cite{Iversen}): If $f$ is a locally
closed inclusion then $f_!$ is exact and has a right adjoint
$f^{(!)}: \Sh X \ra \Sh Y$, which is left exact and preserves
injectives; its right derived $Rf^{(!)}$ yields the right adjoint of
$f_!$ on the corresponding derived categories and is denoted by
$f^!$, as usual. For open inclusions $j$ we have $j^* = j^{(!)}$,
and for closed inclusions $i$, $i_! = i_*$.

\subsection{Stratified spaces and the constructible derived
category}\label{TopologicalConventions} Let $X$ be a topological
space. By a stratification on $X$ we mean a finite stratification
such that
\begin{enumerate}\item[(i)] the strata are simply connected topological
manifolds; \item[(ii)] they are locally closed in $X$; and
\item[(iii)] the closure of any stratum is again a union of
strata.\end{enumerate}

The central objects of interest in this paper are constructible
sheaves and the corresponding derived categories: A sheaf on a
stratified topological space $X$ is called weakly constructible if
its restriction to every stratum is a constant sheaf, and
constructible, if in addition all its stalks are finitely generated.
The usual definition of constructibility is slightly different in
that it requires a sheaf to be locally constant along the strata.
However, since we only allow simply connected strata the two
definitions coincide. We denote the corresponding categories of
constructible and weakly constructible sheaves by $\XCons$ and
$\wXCons$, respectively.

A complex of sheaves is called (weakly) constructible if its
cohomology sheaves are (weakly) constructible. For $*=+$ or $b$, we
denote by $\Dsternc$ the full triangulated subcategories of $D^*(X)$
that consist of constructible complexes, and by $\Dsternwc$ the
corresponding weakly constructible category.

\subsection{The perfect category, quasi-equivalences and formality}
We assume the reader to be familiar with the concepts of dg agebra,
dg module and the corresponding derived categories. Otherwise all
that we need can be found in \cite{BL94}.

For a dg algebra $\mc{A}$, we denote by $\mc{M}_\mc{A}$ the category
of (left) dg modules over $\mc{A}$, and by $\mc{D}_\mc{A}$ the
corresponding derived dg module category. The perfect category
$\Per{\mc{A}}$ is the full triangulated subcategory of
$\mc{D}_\mc{A}$ generated by the direct summands of $\mc{A}$.

A natural question is under which circumstances the perfect
categories of two dg algebras are equivalent. The easiest case is
that of two quasi-isomorphic dg algebras, for which \cite{BL94} have
shown that the corresponding perfect categories are equivalent.
However, there is a more general result, for which we need the
notion of quasi-equivalence. Two dg algebras $\mc{A}$ and $\mc{B}$
are called quasi-equivalent if there is an
$\mc{A}$-$\mc{B}$-dg-bimodule $M$, together with a cycle of degree
zero $c\in M$, such that the cohomology class of $c$ is a basis of
$HM$ as $H\mc{A}$-left- and $H\mc{B}$-right-module. This means that
the morphisms
\begin{eqnarray*}
\mc{A} \hspace{-2mm} &\ra M, \hspace{4mm} a \hspace{-2mm} &\mapsto a\cdot c  \hspace{4mm}\text{ and}\\
\mc{B} \hspace{-2mm} &\ra M, \hspace{4mm} b \hspace{-2mm} &\mapsto
c\cdot b
\end{eqnarray*}
induce isomorphisms on the cohomology level. \cite{Keller} proved
that for two quasi-equivalent dg algebras $\mc{A}$ and $\mc{B}$ the
corresponding perfect categories $\Per{\mc{A}}$ and $\Per{\mc{B}}$
are equivalent.

This result leads us to the definition of formality. Usually a dg
algebra is called formal if it is quasi-isomorphic (in a generalized
sense) to its cohomology algebra, i.e. if we have  a chain of
quasi-isomorphisms
\begin{equation*}
\mc{A} \ra \mc{A}_1 \leftarrow \mc{A}_2 \ra \dots \leftarrow H\mc{A}
\end{equation*}
Considering the result just cited, we would like to introduce a more
general notion of formality: We call a dg algebra $\mc{A}$ formal if
it is (in a generalized sense) quasi-equivalent to its cohomology
algebra, i.e. if we  find dg algbras $\mc{A}=\mc{A}_1, \mc{A}_2,
\dots \mc{A}_{n}= H\mc{A}$ such that $\mc{A}_i$ and $\mc{A}_{i+1}$
are quasi-equivalent. It follows that for a formal dg algebra
$\mc{A}$ the categories $\Per{\mc{A}}$ and $\Per{H\mc{A}}$ will be
equivalent.

\subsection{Derived categories as dg module categories}

We would now like to describe the constructible derived category by
an equivalent dg module category. Before we get started we would
like to remind the reader that for any complex $\mc{F}$ of objects
in an $\Ring$-linear abelian category, the corresponding
endomorphism complex $\End(\mc{F})$ has the structure of a dg
algebra in a natural way. Moreover, if we take two complexes
$\mc{F}$ and $\mc{G}$ and denote by $\Hom(\mc{F},\mc{G})$ the usual
homomorphism complex (as can be found, for example, in \cite{KS}),
$\Hom(\mc{F},\mc{G})$ will be a left dg module over $\End(\mc{G})$
and a right dg module over $\End(\mc{F})$. The following result,
which is due to \cite{Keller}, describes a general method for making
the transition from subcategories of a derived category that are
given by generators, to dg module categories.

\begin{proposition}\label{SoergelsTheorem}
Let $\mc{C}$ be an $\Ring$-linear abelian category, $K^b(\mc{C})$
the category of bounded complexes over $\mc{C}$ and $I$ a bounded
complex of injectives in $\mc{C}$. Then the contravariant functor
\begin{equation}\label{Functor}
\Hom (\cdot,I): K^b(\mc{C}) \ra \mc{M}_{\End I}
\end{equation}
is fully faithful on the full triangulated subcategory of
$D^b(\mc{C})$ that is generated by the direct summands of $I$.
\end{proposition}

Denote by $\langle I \rangle_\oplus$ the full triangulated
subcategory of $D^b(\mc{C})$ that is generated by the direct
summands of $I$. Following \cite{Lunts}, we denote by
$\mc{D}^f_{\End I}$ the image of $\langle I \rangle_\oplus$ under
the functor (\ref{Functor}). Hence $\mc{D}^f_{\End I}$ is the
subcategory of $(\Per{\End I})^\circ$ generated by those direct
summands of $\End I$ that arise from direct summands of $I$, i.e. by
dg modules of the form $(\End I) \cdot p$, where $p$ is a projector
of $I$. Now Proposition \ref{SoergelsTheorem} can be reformulated as
follows:

\begin{corollary}\label{furuns}
The functor (\ref{Functor}) induces an equivalence $\langle I
\rangle_\oplus\simeq \mc{D}^f_{\End I}$.
\end{corollary}

In order to apply this result to $\Dbc$ we need a set of generators
of this category: Denote for any locally closed subset $Z \subset X$
by $\Ring_Z^X$ the constant sheaf on $Z$ extended by zero, i.e. the
sheaf $i_!\cs\Ring Z$, where $i:Z\hra X$ is the inclusion. It is a
well known fact that $\Dbc$ is generated as a triangulated
subcategory of $D^b(X)$ by its heart, i.e. by $\XCons$. This can be
shown using d\'evissage. As a consequence, if our base ring $\Ring$
is a principal ideal domain, $\Dbc$ is generated by
$$\big\{\Ring_\ol{S}^X\left[d(S)\right]\mid S\text{ varies over the
strata}\big\}$$ where $\left[d(S)\right]$ denotes the shift by
$d(S):=\lfloor \frac{\dim S}{2}\rfloor$. If we take the direct sum
of these generators and an injective resolution $I$ of this direct
sum, Corollary \ref{furuns} yields an equivalence
\begin{equation}\label{equivalence}
\Dbc \sira \mc{D}^f_{\End I}
\end{equation}
This is the central result that will allow us to describe the
constructible derived category as a dg module category. However,
injective resolutions are difficult to compute in general. We would
therefore like to allow more general resolutions in Proposition
\ref{SoergelsTheorem}, which leads us to the following result:

\begin{proposition}\label{hinreichendAzyklisch}
Let $\mc{C}$ be an $\Ring$-linear abelian category with enough
injectives. Let $\Omega$ be a bounded complex in $D^b(\mc{C})$ such
that $\Ext^i(\Omega^p,\Omega^q)=0$ for $i>0$ and $p,q\in\BZ$, and
$I$ an injective resolution of $\Omega$ that is bounded from below.
Then $\End\Omega$ and $\End I$ are quasi-equivalent.
\end{proposition}

\begin{proof}
Let $c:\Omega\ra I$ be a quasi-isomorphism. It is a standard result
of homological algebra that
\begin{equation*}
\End I \ra \Hom(\Omega, I), \hspace{3mm} a \mapsto a \circ c
\end{equation*}
is a quasi-isomorphism (\cite{Iversen}). Now we only need to prove
that
\begin{equation*}
\End \Omega \ra \Hom(\Omega, I), \hspace{3mm} b \mapsto c \circ b
\end{equation*}
is a quasi-isomorphism as well. This statement is a generalization
of the well-known fact that, for a left-exact abelian functor $F$,
its right derived $RF$ can be calculated using $F$-acyclic objects
(\cite{KS}). Using this result we get that for every $p\in\BZ$ the
map
\begin{equation*}
\Hom(\Omega^p, \Omega) \ra \Hom(\Omega^p, I)
\end{equation*}
induced by $c$ is a quasi-isomorphism. Now consider the
distinguished triangle
\begin{equation*}
\Omega \overset{c}\ra I \ra M \overset{[1]}\ra,
\end{equation*}
where $M$ is the mapping cone of $c$. Using the distinguished
triangle
\begin{equation*}
\Hom(\Omega^p, \Omega) \ra \Hom(\Omega^p, I) \ra \Hom(\Omega^p, M)
\overset{[1]}\ra
\end{equation*}
we see that $\Hom(\Omega^p, M)$ must have been an exact complex. If
we can show that $\Hom(\Omega, M)$ is exact, we are done since we
have another distinguished triangle
\begin{equation*}
\Hom(\Omega, \Omega) \ra \Hom(\Omega, I) \ra \Hom(\Omega, M)
\overset{[1]}\ra
\end{equation*}
But $\Hom(\Omega,M)$ is the product total complex of the double
complex $\Hom(\Omega^p, M^q)$, which is bounded from left and below
and whose columns are exact, as seen above. For this kind of double
complexes it has been shown in \cite{KS} that their product total
complex is exact.
\end{proof}

\section{Formality for simple stratifications on manifolds}\label{Kap3}

\subsection{The idea}
Let us consider the simplest stratification possible, the
stratification that consists of one stratum only. Assume that our
space is a simply connected differential manifold. We then know that
for real coefficients, $\Dbc$ is described by the dg algebra $\End
I$, where $I$ is an injective resolution of the constant sheaf
$\cs\BR X$. Using that the functor of global sections on $X$ is
naturally equivalent to the functor $\Hom(\cs\BR X, \cdot)$ we see
that $\End I$ is quasi-isomorphic to $R\Gamma(X,\cs \BR X)$. The
latter complex calculates the cohomology of the space $X$, and it is
a standard argument of cohomology theory that it is quasi-isomorphic
to the complex of global sections of the de Rham complex $\deRhama$,
which in turn in certain circumstances is known to be formal (see,
for instance, \cite{DGMS}).

The de Rham complex is much easier to handle than an arbitrary
injective resolution of the constant sheaf. Hence it appears to be a
sensible idea to describe $\Dbc$ by differential forms even if the
stratification is no longer trivial. We will pursue this idea in
this section and deduce a formality result for $n$-spheres at the
end.

\subsection{A generalized de Rham algebra}
Let $X$ be a second-countable differentiable manifold $X$ (which
implies that it is para-compact), stratified in a point $\pt$ and
its complement. Denote by $i_{\pt}: X\backslash\pt \ra X$ the
inclusion. Assume that our base ring is the field of real or complex
numbers, which we denote by $\field$.

If $X\backslash\pt$ is simply connected, we know that $\Dbc$ is
generated by the skyscraper $i_{\pt !}\cs\field\pt$ and the constant
sheaf on $X$, shifted by $\lfloor\frac{\dim X}{2}\rfloor$. However,
shifting the constant sheaf does not affect any of the arguments in
this section, hence for simplicity of notation we drop the shift,
and use the constant sheaf $\cs\field X$ as generator. We then know
from the equivalence \eqref{equivalence} that $\Dbc$ is described by
the dg algebra $\End(I\oplus i_{\pt !}\cs\field\pt)$, where $I$ is
an injective resolution of $\cs\field X$. Now we would like to find
a quasi-equivalent description of $\End(I\oplus i_{\pt
!}\cs\field\pt)$ using differential forms. The first problem we need
to tackle is that the stratum $\pt$ is too small to allow for
interesting differential forms. Hence we need to replace it by a
small ball to clear some space. Thus for the remainder of this
section, fix a closed ball $D$ around $\pt$, i.e. a small closed
neighbourhood that is mapped by a chart to a closed ball in
$\BR^{\dim X}$. Let $i:D\hra X$ be the inclusion. By abuse of
notation, we use $i_{\pt}$ both for the inclusions $\pt\hra D$ and
$\pt\hra X$.

\begin{proposition}\label{Aufblasen}
Choose an injective resolution $J$ of the constant sheaf $\cs\field
D$. Then the dg algebras $\End(I\oplus i_{\pt !}\cs\field\pt)$ and
$\End(I\oplus i_!J)$ are quasi-equivalent.
\end{proposition}

Before we dig into the proof, I would like to recall some standard
facts that we are going to use often throughout this section (for
proofs see \cite{Iversen}). The first one we have already mentioned:
The global section functor $\Gamma(X,\cdot)$ of sheaves of
$\Ring$-modules on a topological space $X$ is naturally equivalent
to the functor $\Hom(\cs\Ring X, \cdot)$. Another useful statement
is that, whenever $I$ is a complex of injectives bounded from below,
every quasi-isomorphism of complexes $\mc{F}$ and $\mc{G}$ yields a
quasi-isomorphism between $\Hom(\mc{F},I)$ and $\Hom(\mc{G},I)$.

\begin{proof}
Consider the operations
\begin{equation*}
\xymatrix{\End \big(I \oplus i_{pt !} \cs\field{pt}\big)
\hspace{0mm} & \ar@(dl,ul) & \hspace{-8mm}\Hom \big(I \oplus i_!J, I
\oplus i_{pt !} \cs\field{pt}\big) \hspace{-8mm}&
  \ar@(dr,ur) &  \hspace{0mm} \End\big(I \oplus i_!J\big)}
\end{equation*}
together with the cycle
$$c:= \begin{pmatrix}id_I & 0\\
0 & \gamma\end{pmatrix} \in \Hom \big(I \oplus i_!J, I \oplus i_{pt
!} \cs\field{pt}\big)$$ where $\gamma$ is the natural extension of
the adjunction morphism $\cs\field D \ra i_{\pt!}\cs\field\pt$,
which is first lifted to $J$ using that the skyscraper is injective,
and is then extended by zero.

We need to show that the following maps induce isomorphisms on
cohomology:
\begin{enumerate}
\item $\Hom(i_{pt !} \cs\field{pt}, I) \ra \Hom(i_!J,I), \,a
\mapsto a\circ \gamma$ \item $\End(i_{pt !} \cs\field{pt}) \ra
\Hom(i_!J,i_{pt !} \cs\field{pt}), \,a \mapsto a\circ \gamma$
\item $\Hom(I,i_!J) \ra \Hom(I,i_{pt !} \cs\field{pt}), \,a
\mapsto \gamma\circ a$ \item $\End(i_!J) \ra \Hom(i_!J,i_{pt !}
\cs\field{pt}), \,a \mapsto \gamma\circ a$
\end{enumerate}
Proving this means a lot of diagram chasing, none of which involves
unexpected twists. The proofs rely on the fact that the inclusion
$X\backslash D \hra X\backslash\pt$ induces isomorphisms on
cohomology, which can be seen using Mayer-Vietoris and the five
lemma. We are only going to prove the first statement and leave the
remaining three, which are easier, to the reader. Else a complete
proof can be found in \cite{Balthasar}.

For (1), we need to show that the corresponding morphism
\begin{equation*}
\Ext^\exponent(i_{\pt !} \cs\field{\pt}, \cs\field X) \ra
\Ext^\exponent(i_!\cs\field D, \cs\field X)
\end{equation*}
is an isomorphism for every $\exponent\in\BZ$ (where as usual
$\Ext^q(\cdot,\cdot)$ denotes the $q^\text{th}$ right derived of
$\Hom(\cdot,\cdot)$). Consider the open inclusions $j_{\pt}:
X\backslash\pt \hra X$ and $j:X\backslash D \hra X$ of the
complements of $\pt$ and $D$, respectively. We get a commutative
diagram of short exact sequences
\begin{equation*}
\xymatrix{j_! j^* \cs\field X \ar[r] \ar[d] & \cs\field X
\ar@{=}[d] \ar[r] & i_!i^* \cs\field X \ar[d]\\
j_{\pt !} j_{pt}^* \cs\field X \ar[r] & \cs\field X \ar[r] & i_{\pt
!} i_{\pt}^* \cs\field X}
\end{equation*}
where the vertical morphisms are induced by the adjuntions
$(j_!,j^*)$ and $(i_\pt^*, i_{\pt !})$, respectively. But since the
inclusion $X\backslash D \hra X\backslash\pt$ induces isomorphisms
on cohomology, we get that concatenation with the left vertical
yields isomorphisms
$$\Ext^\exponent(j_{pt !} j_{pt}^*
\cs\field X, \cs\field X) \sira \Ext^\exponent(j_! j^* \cs\field X,
\cs\field X)$$ Applying the functor $\Ext(\cdot,\cs\field X)$ to our
pair of short exact sequences yields the corresponding commutative
ladder, from which the required result follows by applying the five
lemma.
\end{proof}

Next, we would like to describe the dg algebra $\End(I\oplus i_!J)$
using differential forms. Denote by $\deRham$ the sheaf of
$\BK$-valued differential forms on $X$, which is a soft resolution
of the constant sheaf. In addition to the inclusion $i:D\hra X$ we
now need the inclusion $o: \int(D)\hra X$, where $\int(D) $ denotes
the open interior of $D$. To avoid notational inflation we write
$\Gamma\mc{F}$ for the global sections of a sheaf $\mc{F}$ on $X$.
We would now like to consider the differential graded matrix algebra
\begin{equation}\label{deRhamdef}
\mc{M}(\deRham) := \begin{pmatrix} \Gamma\deRham & \Gamma o_!o^*\deRham \\
\Gamma i_!i^*\deRham & \Gamma i_!i^*\deRham\end{pmatrix}
\end{equation}
whose multiplicative structure is induced by the wedge product. How
this generalized wedge product on $\mc{M}(\deRham)$ works should be
clear if we think of $\Gamma o_!o^*\deRham$ as the sections of
$\deRham$ with support in $\int(D)$ and of $\Gamma i_!i^*\deRham$ as
the sections over $D$, which can be extended to $X$ since the sheaf
of differential forms $\deRham$ is soft.

Using the same point of view we see that there is a natural
inclusion
\begin{equation*}
\begin{pmatrix} \Gamma\deRham & \Gamma o_!o^*\deRham \\
\Gamma i_!i^*\deRham & \Gamma i_!i^*\deRham\end{pmatrix} \hra
\begin{pmatrix} \End \deRham & \Hom(i_!i^*
\deRham, \deRham)\\
\Hom(\deRham, i_!i^*\deRham) & \End i_!i^*\deRham \end{pmatrix}.
\end{equation*}
This inclusion is again induced by the wedge product, more precisely
by the map $a\mapsto a\wedge\cdot$. Since this inclusion is
obviously compatible with the multiplicative structures, we have
proven that $\mc{M}(\deRham)$ has a canonical structure of
dg-sub-algebra of
\begin{equation*}\begin{pmatrix} \End
\deRham & \Hom(i_!i^* \deRham, \deRham)\\
\Hom(\deRham, i_!i^*\deRham) & \End i_!i^*\deRham \end{pmatrix} =
\End(\deRham \oplus i_!i^*\deRham).
\end{equation*}

\begin{proposition}\label{deRham}
For any injective resolution $J$ of the constant sheaf $\cs\field
D$, the dg algebra $\mc{M}(\deRham)$ operates as a dg-sub-algebra of
$\End(\deRham \oplus i_!i^*\deRham)$ canonically on $\Hom(\deRham
\oplus i_!i^*\deRham, I \oplus i_!J)$. This operations yields a
quasi-equivalence between $\mc{M}(\deRham)$ and $\End(I \oplus
i_!J)$.
\end{proposition}

\begin{proof}
The morphisms $\cs\field X\hra I$ and $\cs\field D\hra J$ can be
lifted in the homotopy category to maps $\ol\imath:\deRham\hra I$
and $\ol\jmath: i^*\deRham \hra J$. We get a quasi-isomorphism
\begin{equation*} c :=
\begin{pmatrix} \ol\imath & 0
\\ 0 & i_!\ol\jmath\end{pmatrix}\in \Hom(\deRham \oplus i_!i^*\deRham, I
\oplus i_!J)
\end{equation*}
which is going to yield our quasi-equivalence. A standard argument
shows that
\begin{equation*}
\End(I \oplus i_!J) \ra \Hom\big(\deRham \oplus i_!i^*\deRham, I
\oplus i_!J\big), \, a \mapsto a \circ c
\end{equation*}
is a quasi-isomorphism, and it remains to show that the same is true
for
\begin{equation*}
\begin{pmatrix} \Gamma\deRham & \Gamma o_!o^*\deRham \\
\Gamma i_!i^*\deRham & \Gamma i_!i^*\deRham\end{pmatrix} \ra
\Hom\big(\deRham \oplus i_!i^*\deRham, I \oplus i_!J\big),\, b
\mapsto c \circ b
\end{equation*}
For this, we need to show that the following maps are
quasi-isomophisms:
\begin{enumerate}
\item $\Gamma\deRham \ra \Hom(\deRham, I), \, a \mapsto \ol\imath \circ
(a \wedge \cdot)$ \item $\Gamma o_!o^*\deRham \ra
\Hom(i_!i^*\deRham, I), \, a \mapsto \ol\imath \circ (a \wedge
\cdot)$
\item $\Gamma i_!i^*\deRham \ra \Hom(\deRham, i_!J), \, a \mapsto
i_!\ol\jmath \circ (a \wedge \cdot)$ \item $\Gamma i_!i^*\deRham \ra
\Hom(i_!i^*\deRham, i_!J), \, a \mapsto i_!\ol\jmath \circ (a \wedge
\cdot)$
\end{enumerate}

(1) A quick one as starter: Consider the commutative diagram
\begin{equation*}
\xymatrix{\Gamma\deRham \ar[d]\ar[r] & \Gamma I \ar[d]^\wr \\
\Hom(\deRham, I) \ar[r] & \Hom(\cs\field X, I)}
\end{equation*}
It is a standard result that the horizontal maps are
quasi-isomorphism (for the upper one we need that $\deRham$ is a
soft, hence $\Gamma$-acyclic resolution of the constant sheaf), as
well as that the right vertical is an isomorphism. Hence the left
map must have been a quasi-isomorphism, too, and we are done.

(2) We would like to remind the reader of the fact that for any
sheaf $\mc{F}$ on $X$, $\Gamma o_!o^*\mc{F}$ are the sections with
support in $\int(D)$, and $\Gamma i_! i^{(!)}\mc{F}$ are the
sections with support in $D$ (where $i^{(!)}$ is the right adjoint
of $i_!$ as explained in section \ref{Conventions}). We get a
natural inclusion
$$
\Gamma o_!o^*\mc{F} \hra \Gamma i_! i^{(!)}\mc{F}
$$
which yields the upper right horizontal map in the following
diagram,
\begin{equation*}
\xymatrix{\Gamma o_!o^*\deRham \ar[d]\ar[r] & \Gamma o_!o^*I\ar[r] & \Gamma i_!i^{(!)}I \ar[d]^\wr\\
\Hom(i_!i^*\deRham, \deRham) \ar[d] && \Hom(\cs\field X, i_!i^{(!)}I) \ar[d]^\wr \\
\Hom(i_!i^*\deRham, I) \ar[rr] & & \Hom(i_!i^*\cs\field X, I)}
\end{equation*}
where all other maps should be obvious. As always, the two right
verticals are isomorphisms (the lower one uses two adjunctions), and
the lower horizontal is a quasi-isomorphism. The functors $o_!$ and
$o^*$ are exact, hence $o_!o^*\deRham \ra o_!o^*I$ is a
quasi-isomorphism, and since soft sheaves on paracompact spaces are
$\Gamma$-acyclic, it induces a quasi-isomorphism $\Gamma
o_!o^*\deRham \ra \Gamma o_!o^*I$. Thus it remains to prove that the
natural map $\Gamma o_!o^*I \ra \Gamma i_!i^{(!)} I$ is a
quasi-isomorphism.

For this, we need to calculate the cohomologies of both sides. It is
easy to check that for a locally closed subset $Z\overset{z}{\hra}
Y$ of a topological space $Y$, such that $Z$ has compact closure, we
have a natural isomorphism of functors $\Gamma(Y, \cdot) \circ z_!
\simeq \Gamma_c(Z,\cdot)$. Since $o^*I$ is an injective resolution
of the constant sheaf $\cs\field{\int D}$ on $\int (D)$, we get
\begin{align*}
H^\exponent(\Gamma o_!o^*I) & = H^\exponent_c(\int(D),
\cs\field{\int(D)})
\simeq \begin{cases} 0, \hspace{1mm}\exponent \neq \dim_\BR X\\
\field, \hspace{1mm}\exponent = \dim_\BR X\end{cases}
\end{align*}
where the last isomorphism is due to Poincar\'e duality.

The $\exponent^{\text{th}}$ cohomology of $\Gamma i_!i^{(!)}I$ is,
as follows from the above diagram, equal to
$\Ext^\exponent(i_!i^*\cs\field X, \cs\field X)$, which in turn is,
as we have seen in the proof of Proposition \ref{Aufblasen},
$\Ext^\exponent(i_{pt !}\cs\field{pt}, \cs\field X)$.
A standard argument shows that $i_{\pt}^{!}\cs\field X \simeq
\cs\field\pt[-\dim X]$, and we get
\begin{align*}
\Ext^\exponent(i_{pt !}\cs\field{pt}, \cs\field X) & =
\Ext^\exponent(\cs\field{pt}, i_{pt}^!\cs\field X) \simeq
\begin{cases} 0, \hspace{1mm}\exponent \neq \dim_\BR X\\
\field, \hspace{1mm}\exponent = \dim_\BR X\end{cases}
\end{align*}
So we only need to check that the morphism $\Gamma o_!o^*I \ra
\Gamma i_!i^*I$ induces a surjection on the $(\dim X)^\text{th}$
cohomology. For this, choose a smaller ball $\hat D$ that is
contained in the interior of $D$, and denote by $\hat \i: \hat D\hra
X$ the inclusion. We see that the inclusion $\Gamma \hat \i_!\hat
\i^*I \ra \Gamma i_!i^*I$ is a quasi-isomorphism. (This can be seen
as in the proof of Proposition \ref{Aufblasen}, where we saw that
both are quasi-isomorphic to $\Hom(i_{\pt !}i_{\pt}^*\cs\field X,
I)$ via compatible quasi-isomorphisms.) However, the
quasi-isomorphism $\Gamma \hat \i_!\hat \i^*I \ra \Gamma i_!i^*I$
splits over $\Gamma o_!o^*I$, hence the inclusion $\Gamma o_!o^*I
\ra \Gamma i_!i^*I$ must induce surjections on cohomology.

(3) Considering the following diagram
\begin{equation*}
\xymatrix{\Gamma i_!i^*\deRham \ar[d]\ar[r]& \Gamma i_!J \ar[d]^\wr\\
\Hom(\deRham, i_!J) \ar[r] & \Hom(\cs\field X, i_!J)}
\end{equation*}
we can use a similar argument to (1).

(4) For this we use the diagram
\begin{equation*}
\xymatrix{\Gamma i_!i^*\deRham \ar[dd]\ar[r] & \Gamma i_!J \ar[d]^\wr\\
& \Hom(\cs\field X, i_!J) \ar[d]^\wr\\
\Hom(i_!i^*\deRham, i_!J) \ar[r] & \Hom(i_!i^*\cs\field X, i_!J)}
\end{equation*}
and again a similar argument.
\end{proof}

Combining the equivalence \eqref{equivalence}, Proposition
\ref{Aufblasen} and Proposition \ref{deRham}, we can summarize the
results from this section in the following theorem:

\begin{theorem}\label{deRhamTheorem}
Let $X$ be a second countable differentiable manifold and $\pt$ a
point in $X$ such that $X\backslash\pt$ is simply connected. For
$\deRham$ the sheaf of $\field$-valued differential forms and $D$
any closed ball around $\pt$, let $\mc{M}(\deRham)$ be as defined in
(\ref{deRhamdef}). Then we have an equivalence of triangulated
categories
$$
\Dbc \simeq \mc{D}^f_{\mc{M}(\deRham)}
$$
where $\mc{D}^f_{\mc{M}(\deRham)}$ is the subcategory of
$(\Per{\mc{M}(\deRham)})^\circ$ corresponding to
$\mc{D}^f_{\End{(I\oplus i_{\pt !}\cs\field\pt)}}$ under the
equivalence of categories
$$
\Per{\End{(I\oplus i_{\pt !}\cs\field\pt)}} \simeq
\Per{\mc{M}(\deRham)}
$$
which is induced by the quasi-equivalence of the dg algebras $\End
(I\oplus i_{\pt !}\cs\field\pt)$ and $\mc{M}(\deRham)$.
\end{theorem}

A short remark to finish: For the sake of readability we have
refrained from allowing rings as coefficients in this section.
However, the reader can check that all statements in this section
hold for rings as well, if we replace the corresponding skyscraper
$i_{\pt !}\cs\Ring\pt$ by an injective resolution, and the sheaf of
differential forms by the sheaf of singular cochains (or indeed by
any other soft resolution of the constant sheaf by a sheaf of dg
algebras whose restrictions to open subsets are again soft). This
allows us to take coefficients in a PID, and with little more
effort, get an analogous result.

\subsection{Formality for spheres}
Denote by $\nKugelschale$ the sphere in $\BR^{n+1}$. For $n\geq 2$,
we consider the stratification of $\nKugelschale$ consisting of a
point $\pt$ and its complement, and now use Theorem
\ref{deRhamTheorem} to prove formality of the corresponding
constructible derived category.

We use the notations from the previous section: For $\BK$ the real
or complex numbers, let $\deRham$ be the sheaf of $\BK$-valued
differential forms on $\nKugelschale$. Choose a small closed ball
$D$ around $\pt$. For $i: D\hra\nKugelschale$ and
$o:\int(D)\hra\nKugelschale$ the inclusions, let $\mc{M}(\deRham)$
be the dg algebra defined in \eqref{deRhamdef}. Theorem
\ref{deRhamTheorem} then provides an equivalence
$$
D^b_c(\nKugelschale) \simeq \mc{D}^f_{\mc{M}(\deRham)}
$$
and we now give the promised formality result:

\begin{theorem}\label{formalityfornspheres}
The generalized de Rham algebra $\mc{M}(\deRham)$ on the $n$-sphere
$S^n$ (for $n\geq 2$) is formal.
\end{theorem}

\begin{proof}
First we need to understand what the cohomology of $\mc{M}(\deRham)$
looks like. The global sections $\Gamma \deRham$ calculate the
cohomology of the sphere:
$$
H^\exponent(\Gamma \deRham) =
\begin{cases}\BK, \hspace{1mm}\exponent = n \text{ or } 0\\ 0,
\hspace{1mm}\text{ otherwise} \end{cases}
$$

As in the proof of Proposition \ref{deRham}, we see that
$$
H^\exponent(\Gamma o_!o^*\deRham) = H^\exponent_c(\int(D),
\cs\BK{\int(D)}) \simeq
\begin{cases}\BK, \hspace{1mm}\exponent = n \\ 0,
\hspace{1mm}\exponent \neq n \end{cases}
$$
and
$$
H^\exponent(\Gamma i_!i^*\deRham)  = H^\exponent(D, \cs\BK D) \simeq
\begin{cases} \BK, \hspace{1mm}\exponent = 0\\ 0, \hspace{1mm}\exponent \neq
0\end{cases}.
$$
We claim that the natural inclusion $\Gamma o_!o^*\deRham \ra
\Gamma\deRham$ yields an isomorphism on the $n^\text{th}$ cohomology
group. This can be seen as follows: Denoting by $h: \nKugelschale
\backslash \int(D) \hra \nKugelschale$ the inclusion, we get a short
exact sequence
\begin{equation*}
o_!o^*\deRham \hra \deRham \tha h_!h^*\deRham
\end{equation*}
Since all of those are complexes of soft, hence $\Gamma$-acyclic
sheaves, we get a corresponding short exact sequence on the global
sections.
Our claim follows from the corresponding long exact cohomology
sequence, if we use that the cohomology groups
\begin{align*}
H^\exponent(\Gamma h_!h^*\deRham) & = H^\exponent(\Gamma(
\nKugelschale \backslash \int(D) , h^*\deRham)) = H^\exponent(
\nKugelschale \backslash \int(D) , \cs\BK{\nKugelschale \backslash
\int(D)})
\end{align*}
vanish for $\exponent > 0$.

Now choose a generator $\omega\in\Gamma o_!o^*\deRham$ of the
$n^\text{th}$ cohomology. By what we have just shown, $\omega$ must
then be a generator of the $n^\text{th}$ cohomology of the sphere
$\nKugelschale$. Since $\Gamma i_!i^*\deRham$ has no $n^\text{th}$
cohomology, the cohomology class of $\omega\vert_D$ must vanish.
Hence there is a $\tau\in\Gamma i_!i^*\deRham^{n-1}$ satisfying
$d\tau = \omega\vert_D$.

Now define a sub-vector space $\mc{U}$ of $\mc{M}(\deRham)$ as
follows:
\begin{equation*}
\mc{U} := \left(\begin{array}{ccccccccccc}
&&\BK\cdot\omega&&&&&&\BK\cdot\omega&&\\
&&\oplus&&&&&&&&\\
&&\BK\cdot 1&&&&&&&&\\
&&&&&&&&&&\\
&&&&&&&&&&\\
&&\BK \cdot \omega\vert_D&&&&&&\BK \cdot \omega\vert_D&&\\
&&\oplus&&&&&&\oplus&&\\
&&\BK\cdot\tau&&&&&&\BK\cdot\tau&&\\
&&\oplus&&&&&&\oplus&&\\
&&\BK\cdot 1&&&&&&\BK\cdot 1&&\\
\end{array}\right)
\end{equation*}
where $\BK\cdot\omega$ denotes the vector space generated by
$\omega$ etc. It is straightforward that $\mc{U}$ is closed under
multiplication: As the de Rham-complex vanishes in degrees greater
than $n$, it suffices to check that $\tau\wedge\tau = 0$. For $n\geq
3$ this is clear for degree reasons, and for $n=2$ it holds because
$\Gamma i_!i^*\deRham$ is super-commutative and the degree of $\tau$
is odd.

From the construction of $\mc{U}$ it is clear that the inclusion
$\mc{U} \hra \mc{M}(\deRham)$ is a quasi-isomorphism. It is also
obvious that the projection $\mc{U} \tha H(\mc{U})$ is a
quasi-isomorphism, hence the chain of quasi-isomorphisms
$$
\mc{M}(\deRham) \hookleftarrow \mc{U} \tha H(\mc{U})
$$
yields the desired formality.
\end{proof}

As mentioned at the beginning of this section, it is straightforward
to check that shifting the constant sheaf on $S^n$ does not affect
any of the arguments in this section. All we need to do is adapt the
definition of $\mc{M}(\deRham)$ by introducing corresponding shifts.
Hence in this section we have effectively proved Theorem \ref{main1}
(ii) from the introduction.

\section{A description of $\Dbc$ by representations of quivers}\label{Kap4}

\subsection{$\mc{S}$-acyclic sheaves and $\mc{S}$-acyclic stratifications}

In this section we would like to describe $\Dbc$ using
representations of quivers. Let $\Ring$ be a commutative unitary
Noetherian ring of finite homological dimension and $X$ a stratified
topological space. We would like to remind the reader that we assume
any stratification to satisfy the conditions (i) - (iii) introduced
in section \ref{TopologicalConventions}. We denote by $\mc{S}$ the
set of strata, and for a stratum $S\in\mc{S}$ by $\Staralpha$ its
star, i.e. the union of all strata whose closure contains $S$:
\begin{equation*}
\Staralpha = \bigcup_{\Str \subset \ol\Strr} \Strr
\end{equation*}
The star of $S$ is the smallest open set consisting of strata that
contains $S$.

Let us now consider the projection $p: X \tha \mc{S}$ that maps
every point to the stratum that contains it. We endow $\mc{S}$ with
the final topology and would now like to link the category of
sheaves on $\mc{S}$, $\Sh{\mc{S}}$, to the category of weakly
constructible sheaves on $X$, $\wXCons$.

\begin{proposition}\label{UnderivedEquivalence}
Let $\mc{S}$ be a stratification of a topological space $X$ that
satisfies the following assumption: For every weakly constructible
sheaf $\mc{F}$, every stratum $S\in\mc{S}$ and every $x\in S$ the
natural morphism
$$\Gamma(\Staralpha, \mc{F}) \ra \mc{F}_x$$
is an isomorphism. Then for $p:X\tha\mc{S}$ the projection, the
adjoint pair $(p^*,p_*)$ yields an equivalence of categories
\begin{equation*}
\wXCons \overset{p_*}{\underset{p^*}\rightleftarrows} \Sh{\mc{S}}
\end{equation*}
\end{proposition}

\begin{proof}
We just need to check that $p^*p_* \ra id$ and $id \ra p_*p^*$ are
natural isomorphisms, which is straightforward by looking at the
stalks.
\end{proof}

We would like to extend this statement to the corresponding derived
categories. For this, we need to introduce the notions of
$\mc{S}$-acyclic sheaves and acyclic stratifications.

\begin{definition}
Let $\mc{S}$ be a stratification of a topological space $X$. We call
a sheaf $\mc{F}$ on $X$ $\mc{S}$-acyclic if, for every stratum
$S\in\mc{S}$ and every $i > 0$,
\begin{equation*}
H^i\big(\Staralpha, \mc{F}) = 0
\end{equation*}
holds. This is equivalent to $\mc{F}$ being $p_*$-acyclic, where
$p:X\ra\mc{S}$ is the projection.
\end{definition}

\begin{definition}
We call a stratification $\mc{S}$ of $X$ acyclic if it satisfies the
following two conditions:
\begin{enumerate}
\item For every weakly constructible sheaf $\mc{F}$, every stratum $S\in\mc{S}$ and
and every $x\in S$ the natural map $\Gamma(\Staralpha, \mc{F}) \ra
\mc{F}_x$ is an isomorphism.
\item Every weakly constructible sheaf on $X$ is $\mc{S}$-acyclic.
\end{enumerate}
\end{definition}

It is proved in \cite{KS} that simplicial complexes with their
natural stratification are acyclic. We will see more examples of
acyclic stratifications in section \ref{Kap5}.

\begin{theorem}\label{weakequivalence}
On a topological space $X$ with an acyclic stratification $\mc{S}$
the adjoint pair $(Rp_*, p^*)$ induces an equivalence of categories
\begin{equation*}
\Dpluswc \underset{p^*}{\overset{Rp_*}\rightleftarrows} D^+(\mc{S}).
\end{equation*}
If $X$ has finite homological dimension, these induce equivalences
\begin{equation*} \Dbwc
\underset{p^*}{\overset{Rp_*}\rightleftarrows} D^b(\mc{S}).
\end{equation*}
\end{theorem}

\begin{proof}
We have to show that $id \ra Rp_*p^*$ and $p^*Rp_* \ra id$ are
natural equivalences. The first case follows from Proposition
\ref{UnderivedEquivalence}, and the fact that $p^*$ transforms a
sheaf on $\mc{S}$ into a $p_*$-acyclic object. For the second case
we need to show that for any complex $\mc{F} \in \Dpluswc$
consisting of $p^*p_*$-acyclic objects the morphism $p^*p_*\mc{F}
\ra \mc{F}$ is a quasi-isomorphism. This, however, works exactly as
the proof of Proposition 8.1.9 in \cite{KS}.
\end{proof}

Unfortunately, the same statement for the constructible derived
category doesn't work quite as smoothly. We need to introduce some
notation first: Let $\mc{C}$ be an abelian category and $\mc{C}'$ a
thick full abelian subcategory. Then $D^+_{\mc{C}'}(\mc{C})$ denotes
the full triangulated subcategory of $D^+(\mc{C})$ consisting of
complexes whose cohomology objects are in $\mc{C}'$.
$D^b_{\mc{C}'}(\mc{C})$ is defined analogously. We are interested in
$\mc{C} = \Sh{\mc{S}}$, the category of sheaves on $\mc{S}$, and
$\mc{C}' = \sh{\mc{S}}$, the full abelian subcategory of sheaves
with finitely generated stalks.

\begin{theorem}
On a topological space $X$ with an acyclic stratification $\mc{S}$
the adjoint pair $(Rp_*, p^*)$ induces an equivalence of categories
\begin{equation*}
D^+_c(X) \underset{p^*}{\overset{Rp_*}\rightleftarrows}
D^+_{\sh{\mc{S}}}(\Sh{\mc{S}}).
\end{equation*}If $X$ has finite homological dimension, these induce equivalences
\begin{equation*}
D^b_c(X) \underset{p^*}{\overset{Rp_*}\rightleftarrows}
D^b_{\sh{\mc{S}}}(\Sh{\mc{S}}).
\end{equation*}
\end{theorem}

\begin{proof}
The proof works as for Theorem \ref{weakequivalence}. Note that this
theorem is still true if we allow stratifications that are only
``almost'' acyclic: The second condition in the definition of an
acyclic stratification needs to be true only for constructible
sheaves.
\end{proof}

This theorem is rounded off by the following statement:
\begin{proposition}
For any stratified topological space the natural functor
\begin{displaymath}
D^b(\sh{\mc{S}})\hspace{3mm}\sira \hspace{3mm}
D^b_{\sh{\mc{S}}}(\Sh{\mc{S}})\hspace{4mm}
\end{displaymath}
yields an equivalence of triangulated categories.
\end{proposition}

\begin{proof}
The trick is to use a standard criterion for an equivalence between
$D^b(\mc{C})$ and $D^b_{\mc{C}'}(\mc{C})$, which can be found in
\cite{KS}. The details work very similarly to the proof of Theorem
8.1.11 in \cite{KS}; a detailed proof can be found in
\cite{Balthasar}.
\end{proof}

\subsection{Representations of quivers}\label{Kap4.2}

We would now like to make the transition from sheaves on $\mc{S}$ to
representations of quivers. A detailed account on quivers can be
found in \cite{GabrielRoiter}; I will only give a very short summary
here.

By a quiver we mean a finite directed graph. A representation $V$ of
a quiver is a collection of $\Ring$-modules $V_x$ for every vertex
$x$ (the stalks), together with $\Ring$-linear maps $V_\gamma: V_x
\ra V_y$ for every arrow $\gamma$ from $x$ to $y$. A map of
representations is a compatible collection of maps on the stalks.
For a quiver $Q$ we denote by $\Rep Q$ the category of
representations of $Q$ and by $\Repee Q$ the category of
representations whose stalks are finitely generated. A relation in a
quiver is an $\Ring$-linear combination of paths starting and ending
at the same vertices. If we have a quiver $Q$ with a set of
relations $\rho$, we denote by $\Rep(Q,\rho)$ the full subcategory
of $\Rep Q$ given by representations that are compatible with the
relations. The category $\Repee(Q,\rho)$ is defined analogously.

To every stratified space $(X,\mc{S})$ we assign a quiver as
follows: The vertices are given by the strata in $\mc{S}$, and there
is an arrow from $S$ to $T$ if and only if $S\subset \ol T$. By
identifying all paths between two vertices we get a set of
relations. We denote the corresponding quiver with relations by
$(Q_\mc{S}, \rho_\mc{S})$.

Now there is a quite obvious functor $\Sh{\mc{S}} \ra \Rep
(Q_\mc{S},\rho_\mc{S})$ and vice versa: From a sheaf $\mc{F}$ on
$\mc{S}$ we get a representation by assigning to each vertex the
corresponding stalk, and to an arrow $S\ra T$ we assign the
restriction map $\Gamma(p(\Staralpha), \mc{F}) \ra
\Gamma(p(\Starbeta), \mc{F})$. This functor is exact and induces
equivalences $\Rep (Q_\mc{S},\rho_\mc{S}) \sira \Sh{\mc{S}}$ and
$\Repee (Q_\mc{S},\rho_\mc{S}) \sira \sh{\mc{S}}$. This is easily
checked using the obvious inverse functor, which is also exact.
Hence for $*=+$ or $b$, these functors yield equivalences
\begin{align*}
D^*({\Sh{\mc{S}}}) \rightleftarrows D^*(\Rep(Q_\mc{S}, \rho_\mc{S}))
\hspace{8mm} & \text{and} & D^*({\sh{\mc{S}}}) \rightleftarrows
D^*(\Repee(Q_\mc{S}, \rho_\mc{S})).
\end{align*}
Summarizing the results from this section, we get the following
theorem:
\begin{theorem}\label{quivers}
We can assign a quiver $Q_\mc{S}$ with relations $\rho_\mc{S}$ to
every acyclic stratification $\mc{S}$ on a topological space $X$ of
finite homological dimension, such that we get natural equivalences
$$
\Dsternwc \sira D^*(\Rep(Q_\mc{S},\rho_\mc{S}))
\hspace{5mm}\text{for $* = +,b$}
$$
and
$$
\Dbc \sira D^b(\Repee(Q_\mc{S},\rho_\mc{S}))
$$
\end{theorem}

\section{Formality for the 2-sphere}\label{Kap5}

We would now like to prove formality for the 2-sphere $\Kugelschale$
stratified in several points and their complement. The clue is to
use Theorem \ref{quivers} to transfer the problem of formality to a
category of representations of quivers. Since we are going to
encounter several stratifications during this chapter, we are going
to carry the relevant stratification in the notation of the
constructible derived category; hence from now on we will write
$D^b_{c,\mc{S}}(X)$ instead of $\Dbc$. For the remainder of this
paper we assume the base ring $\Ring$ to be a principal ideal
domain.

Denote by $\Sn$ the stratification of the 2-sphere in $n$ points and
their complement (this stratification is not unique, but it does not
matter for our purposes which points we choose). $\Snull$ denotes
the trivial stratification. We will need to ``acyclify'' those
stratifications; hence we would now like to introduce certain
acyclic stratifications of the 2-sphere: Take $n$ points ($n \geq
2$) on a great circle on the sphere, and stratify the sphere into
those points, the circle segments between them and the two
hemi-spheres. We denote this stratification by $\An$. This
stratification is acyclic; the proof is not difficult and can be
found in \cite{Balthasar}. Obviously it is not crucial that the
points are on a great circle; if they are not, it is still be
possible to connect them via a deformed circle, and the resulting
stratification will still be acyclic.

\subsection{The trivial stratification}\label{Kap5.1} First, we would like to consider the trivial
stratification $\Snull$  on $\Kugelschale$ and the corresponding
constructible derived category $D^b_{c,\Snull}(\Kugelschale)$. We
know that this category is generated by the constant sheaf
$\cs\Ring\Kugelschale$. It follows from \eqref{equivalence} that for
any injective resolution $I$ of the constant sheaf,
$D^b_{c,\Snull}(\Kugelschale)$ is equivalent to $\mc{D}^f_{\End I}$.
We will now show that the dg algebra $\End I$ is formal:

\begin{proposition}\label{trivialstratification}
Let $I$ be an injective resolution of the constant sheaf
$\cs\Ring\Kugelschale$ on the 2-sphere. Then $\End I$ is
quasi-equivalent to $\Ring[t]/t^2$, where t lives in degree 2. In
particular, $\End I$ is formal.
\end{proposition}

The standard proof for this statement would exploit the fact that
$\End I$ calculates the cohomology of the 2-sphere, and to directly
give a quasi-isomorphism $H(\End I) \ra \End I$. However, we would
like to give a different proof as a preparation for more complex
cases. We will break down the proof into several statements; the
crucial part is Proposition \ref{formalitytrivialstratification}.

First we need to acyclify the trivial stratification. For this we
take the stratification $\Azwei$ (as introduced at the beginning of
section \ref{Kap5}), which consists of two points, two hemi-equators
and two hemi-spheres. We denote the two points by $P_1$ and $P_2$,
the hemi-equators by $E_1$ and $E_2$, and the hemi-spheres by $H_1$
and $H_2$.

The quiver $(Q_{\Azwei}, \rho_{\Azwei})$ associated to this
stratification, as introduced in section \ref{Kap4.2}, is given by
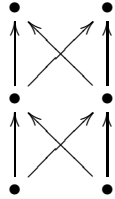
\begin{equation*} \xymatrix{\bullet & \bullet\\
\bullet \ar[u] \ar[ur]& \bullet \ar[u] \ar[ul]\\ \bullet \ar[u]
\ar[ur] & \bullet \ar[u] \ar[ul]}
\end{equation*}
with relations that identify all paths that have the same starting
and end point. In the diagram above, the two upper vertices
correspond to the hemi-spheres, the middle ones to the hemi-equators
and the lower ones to the two points in the stratification. By
definition of $(Q_{\Azwei}, \rho_{\Azwei})$ we should also have
arrows from the lower to the upper vertices; however, due to the
relations we can drop those.

We would now like to transfer our problem to a category of
representations of this quiver, using the natural equivalences of
categories
$$\Rep(Q_{\Azwei},\rho_{\Azwei}) \sira
\Sh{\Azwei} \sira
\operatorname{Sh}_{w,\hspace{0,5mm}c,\hspace{0,5mm}
\Azwei}(\Kugelschale)
$$
that we constructed in the previous section. Under these
equivalences, the representation corresponding to the constant sheaf
$\cs\Ring\Kugelschale$ is
\begin{equation*}
\xymatrix{& \Ring & \Ring\\
C \hspace{4mm}:= & \Ring \ar[u] \ar[ur]& \Ring \ar[u] \ar[ul]\\
& \Ring \ar[u] \ar[ur] & \Ring \ar[u] \ar[ul]}
\end{equation*}
with all arrows carrying the identity map.

\begin{lemma}\label{sufficientlyacyclic}
Let $J$ be a bounded resolution of the representation $C$ that is
sufficiently acyclic in the sense that $\Ext^\exponentzwei(J^m, J^n)
= 0$ for all $\exponentzwei > 0$ and all $m,n \in \BZ$. Then our dg
algebra $\End I$ from Proposition \ref{trivialstratification} is
quasi-equivalent to the dg algebra $\End J$.
\end{lemma}

\begin{proof}
Using that the concatenation $\Rep(Q_{\Azwei},\rho_{\Azwei}) \ra
\Sh{\Azwei} \ra \operatorname{Sh}_{w,\hspace{0,5mm}c,\hspace{0,5mm}
\Azwei}(\Kugelschale)$ is exact and fully faithful, the result
follows by applying Proposition \ref{hinreichendAzyklisch}.
\end{proof}

The next thing to do is resolving $C$ by sufficiently acyclic
objects. As sufficiently acyclic objects, we would like to take, for
every stratum $S$ of $\Azwei$, the representation $I_S$
corresponding to the constant sheaf along the closure of $S$. To be
more precise, for $i:\ol{S}\hra\Kugelschale$ the inclusion, $I_S$ is
the image of $i_!\cs\Ring{\ol{S}}$ under the equivalence
$\operatorname{Sh}_{w,c,\Azwei} \overset\sim\ra
\Rep(Q_{\Azwei},\rho_{\Azwei})$. This procedure yields the following
representations of $(Q_{\Azwei},\rho_{\Azwei})$:
\begin{equation*}
\xymatrix{& \Ring & 0\\
I_{H_1} = & \Ring \ar[u] \ar[ur]& \Ring \ar[u] \ar[ul]\\
& \Ring \ar[u] \ar[ur] & \Ring \ar[u] \ar[ul]} \hspace{10mm}
\xymatrix{& 0 & \Ring\\
I_{H_2} = & \Ring \ar[u] \ar[ur]& \Ring \ar[u] \ar[ul]\\
& \Ring \ar[u] \ar[ur] & \Ring \ar[u] \ar[ul]}
\end{equation*}
\begin{equation*}
\xymatrix{& 0 & 0\\
I_{E_1} = & \Ring \ar[u] \ar[ur]& 0 \ar[u] \ar[ul]\\
& \Ring \ar[u] \ar[ur] & \Ring \ar[u] \ar[ul]} \hspace{10mm}
\xymatrix{& 0 & 0\\
I_{E_2} = & 0 \ar[u] \ar[ur]& \Ring \ar[u] \ar[ul]\\
& \Ring \ar[u] \ar[ur] & \Ring \ar[u] \ar[ul]}
\end{equation*}
\begin{equation*}
\xymatrix{& 0 & 0\\
I_{P_1} = & 0 \ar[u] \ar[ur]& 0 \ar[u] \ar[ul]\\
& \Ring \ar[u] \ar[ur] & 0 \ar[u] \ar[ul]} \hspace{10mm}
\xymatrix{& 0 & 0\\
I_{P_2} = & 0 \ar[u] \ar[ur]& 0 \ar[u] \ar[ul]\\
& 0 \ar[u] \ar[ur] & \Ring \ar[u] \ar[ul]}
\end{equation*}
where on the arrows we have the identity on $\Ring$, wherever
possible, and the the zero map everywhere else. These objects are
indeed sufficiently acyclic:

\begin{lemma}\label{keineerweiterungen}
For arbitrary strata $S,T \in \Azwei$ and every $q > 0$ we have that
$$\Ext^\exponentzwei(I_\Str, I_\Strr) = 0.$$
\end{lemma}

\begin{proof}
It suffices to show the same statement for the corresponding
sheaves. The explicit arguments can be found in \cite{Balthasar}.
\end{proof}

Next, we need to consider the morphisms between our sufficiently
acyclic objects. The morphism spaces between the various $I_S$ are
free $\Ring$-modules. They vanish in all cases but the following, in
which they are free of rank 1:
\begin{alignat*}{2} &\Hom(I_{H_j}, I_{H_j}) & \qquad\qquad
& \text{ for } j = 1,2;\\
&\Hom(I_{H_j}, I_{E_i}) & \qquad\qquad & \text{ for } i = 1,2 \text{
und }j = 1,
2;\\
&\Hom(I_{H_j}, I_{P_i}) & \qquad\qquad & \text{ for } i = 1,2 \text{
und
} j = 1,2;\\
&\Hom(I_{E_i}, I_{E_i}) & \qquad\qquad & \text{ for } i = 1,2;\\
&\Hom(I_{E_j}, I_{P_i}) & \qquad\qquad & \text{ for } i = 1,2 \text{
und }j = 1,
2;\\
&\Hom(I_{P_i}, I_{P_i}) & \qquad\qquad & \text{ for } i = 1,2.
\end{alignat*}

Each of these morphism spaces has a canonical generator, which is
given by the identity on $\Ring$, wherever possible, and zeros
everywhere else. The generator of $\Hom(I_{H_i}, I_{H_i})$ is
denoted by $h_i$, analogously we denote by $e_i$ and $p_i$ the
generators of $\Hom(I_{E_i},I_{E_i})$ and $\Hom(I_{P_i}, I_{P_i})$,
respectively. The generator of $\Hom(I_{H_j}, I_{E_i})$ is denoted
by $\he{ji}$, that of $\Hom(I_{H_j}, I_{P_i})$ by $\hp{ji}$ and that
of $\Hom(I_{E_j}, I_{P_i})$ by $e_{ji}$. \footnote{The underlying
principle of these notations is as follows: If a morphism starts at
$I_{H_j}$, we use $h$ for notation, if it starts at $I_{E_j}$, we
use $e$. In most cases the generator will be determined uniquely by
the starting representation and the indices used; the only
exceptions are the generators of $\Hom(I_{H_j}, I_{E_i})$ and
$\Hom(I_{H_j}, I_{P_i})$, where we add the range as a superscript.}

Consider now the following resolution $J$ of $C$:
\begin{equation}\label{resolutiontrivialstrat}
C \hra I_{H_1} \oplus I_{H_2} \overset{\begin{pmatrix} \he{11} & -\he{21} \\
\he{12} & -\he{22}
\end{pmatrix}}\ra I_{E_1} \oplus I_{E_2} \overset{\begin{pmatrix}
e_{11}& -e_{21}\\ -e_{12} & e_{22}
\end{pmatrix}}\ra
I_{P_1} \oplus I_{P_2}
\end{equation}
Exactness of the sequence is easily checked on stalks.

\begin{proposition}\label{formalitytrivialstratification}
The endomorphism-dg-algebra $\End J$ of the above resolution of $C$
is quasi-isomorphic to its cohomology.
\end{proposition}
This result is the essential step in the proof of Proposition
\ref{trivialstratification}: By Lemma \ref{sufficientlyacyclic},
$\End I$ is quasi-equivalent to $\End J$, which in turn is
quasi-isomorphic to its cohomology by the above proposition, and the
cohomology is easily seen to be $\Ring[t]/t^2$ (where $t$ lives in
degree 2). So it remains to prove Proposition
\ref{formalitytrivialstratification}.

\begin{proof}
This proof is based on a hands-on calculation of $\End J$ and its
cohomology. For shortness of notation we write $\mc{E}$ for $\End
J$. Since $\Hom(\cdot,\cdot)$ commutes in both entries with finite
direct sums, we get that for any $m\in\BZ$, $\mc{E}^m$ is a direct
sum of free $\Ring$-modules, which we would now like to give a basis
of. Before we get down to this, however, we would like to fix the
following notational conventions: Morphisms between direct sums are
written as matrices, as is the usual convention. However, matrices
of the form
$\begin{pmatrix} g & 0 \\
0 & 0 \end{pmatrix}$, $\begin{pmatrix} 0 & g \\
0 & 0 \end{pmatrix}$ etc. are abbreviated to $g$, as are tupels of
the form $(\dots 0, g , 0\dots) \in \prod\limits_{p \in \BZ}
\Hom(I^p, I^{p+m})$. It should always be clear from context which
morphism is meant.

The following table gives a basis of $\mc{E}$, and the image of the
basis elements under the differential $d_{\mc{E}}$.

\begin{longtable}{|l|lcl|}
\nopagebreak\hline m & basis of $\mc{E}^m$& & image under
$d_\mc{E}$\\\hline\endhead\nopagebreak \hline \endfoot\endlastfoot 0
& $h_1$ & $\mapsto$ & $\he{11} + \he{12}$\\* &$h_2$ & $\mapsto$ &
$-\he{21} -\he{22}$\\* &$e_1$ & $\mapsto$ & $(\he{21}-\he{11}) +
(e_{11} - e_{12})$\\* &$e_2$ & $\mapsto$ & $(\he{22}-\he{12}) +
(e_{22} - e_{21})$\\* &$p_1$ & $\mapsto$ & $e_{21}-e_{11}$\\* &$p_2$
& $\mapsto$ & $e_{12}-e_{22}$
\\ \hline 1 & $\he{11}$ & $\mapsto$ & $\hp{11} - \hp{12}$\\*
&$\he{12}$ & $\mapsto$ & $\hp{12} - \hp{11}$\\ &$\he{21}$ &
$\mapsto$ & $\hp{21} - \hp{22}$\\ &$\he{22}$ & $\mapsto$ & $\hp{22}
- \hp{21}$\\* &$e_{11}$ & $\mapsto$ & $\hp{11} - \hp{21}$\\
&$e_{12}$ & $\mapsto$ & $\hp{12} - \hp{22}$\\* &$e_{21}$ & $\mapsto$
& $\hp{11} -
\hp{21}$\\* &$e_{22}$ & $\mapsto$ & $\hp{12} - \hp{22}$\\
\hline 2 & $\hp{11}$ & $\mapsto$ & $0$\\* &$\hp{12}$ & $\mapsto$ &
$0$\\* &$\hp{21}$ & $\mapsto$ & $0$\\* &$\hp{22}$ & $\mapsto$ &
$0$\\\hline
\end{longtable}
For $m \neq 0,1,2$, we have $\mc{E}^m= 0$. Next, we need to
calculate the cohomology of $\mc{E}$; please note that we are not
going to differentiate notationally between an element of the kernel
of $d_\mc{E}$ and the corresponding element in the cohomology; it
should always be clear from context what is meant.

It is easy to see that $1$ yields a basis of $\ker d_\mc{E}^0$ and
hence a basis of $H^0\mc{E}$. A basis of $\ker d_\mc{E}^1$ is given
by
\begin{equation*} \{d_\mc{E}^0(f) \vert f \text{ passes through the above
basis of } \mc{E}^0 \text{ except for } h_2\}
\end{equation*}
which also generates $\im d_\mc{E}^0$, hence $H^1\mc{E}$ vanishes.
Finally, a basis of $\im d_\mc{E}^2$ is given by
\begin{equation*}
\{d_\mc{E}^1(\he{11}), d_\mc{E}^1(\he{21}), d_\mc{E}^1(e_{11})\},
\end{equation*}
while a basis of $\ker d_\mc{E}^2$ is given by our basis of
$\mc{E}^2$. Accordingly, $H^2\mc{E}$ is free of rank 1 and is
generated by $\hp{11}$.\\
It is easy to see that the canonical inclusion $\ker d^0 \hra
\mc{E}^0$, and the morphism $H^2\mc{E}\ra \mc{E}^2$ that maps the
cohomology class of $\hp{11}$ to its representant $\hp{11}$, combine
to a quasi-isomorphism of dg algebras $H\mc{E} \ra \mc{E}$, and we
are done.
\end{proof}

\subsection{The 2-sphere and a point} Next, we would like to discuss
the 2-sphere stratified in a point $\pt$ and its complement. We had
denoted this stratification by $\Seins$. The corresponding
generators of $D^b_{c,\Seins}(\Kugelschale)$ are the skyscraper at
$\pt$ and the constant sheaf on the sphere, shifted by 1.


Hence we now need to choose an injective resolution $I$ of
$\mc{W}_{\pt}\oplus\cs\Ring\Kugelschale[1]$, where $\mc{W}_{\pt}$
denotes the skyscraper at $\pt$. As previously, by
\eqref{equivalence} we then have an equivalence
$D^b_{c,\Seins}(\Kugelschale) \simeq \mc{D}^f_{\End I}$. Hence the
following theorem, which shows that $\End I$ is formal, concludes
the proof of Theorem \ref{main1} (i) from the introduction.

\begin{theorem}\label{onepointstratification}
Denote by $\pt$ the point in the stratification $\Seins$ and let
$\mc{W}_{\pt}$ be the skyscraper at this point. For any injective
resolution $I$ of $\mc{W}_{\pt}\oplus\cs\Ring\Kugelschale[1]$ the
corresponding endomorphism-dg-algebra $\End I$ is formal.
\end{theorem}

\begin{proof}
The proof of this statement will be similar to the proof of
Proposition \ref{trivialstratification}. Again we need to acyclify
the stratification $\Seins$. For this we take the acyclic
stratification $\Azwei$, where one of the points of $\Azwei$ is
taken to be $\pt$. Under the equivalence
$$\Rep(Q_{\Azwei},\rho_{\Azwei}) \sira
\Sh{\Azwei} \sira
\operatorname{Sh}_{w,\hspace{0,5mm}c,\hspace{0,5mm}
\Azwei}(\Kugelschale)
$$
the representations corresponding to the skyscraper and the constant
sheaf on $\Kugelschale$ are
\begin{equation*}
\xymatrix{& 0 & 0\\
W:= & 0 \ar[u] \ar[ur]& 0 \ar[u] \ar[ul]\\
& \Ring \ar[u] \ar[ur] & 0 \ar[u] \ar[ul]} \hspace{20mm}
\xymatrix{& \Ring & \Ring\\
C:= & \Ring \ar[u] \ar[ur]& \Ring \ar[u] \ar[ul]\\
& \Ring \ar[u] \ar[ur] & \Ring \ar[u] \ar[ul]}
\end{equation*}

As in section \ref{Kap5.1}, we need to find a sufficiently acyclic
resolution of $W \oplus C[1]$; we can do this using the same objects
as in \eqref{resolutiontrivialstrat}, and get a resolution
\begin{equation*}
I_{H_1} \oplus I_{H_2} \overset{\begin{pmatrix}
\he{11} & -\he{21}\\
\he{12} & -\he{22}\\
0 & 0
\end{pmatrix}}\ra I_{E_1} \oplus I_{E_2} \oplus I_{P_1} \overset{\begin{pmatrix}
e_{11} & -e_{21} & 0\\ -e_{12} & e_{22} & 0
\end{pmatrix}}\ra
I_{P_1} \oplus I_{P_2}
\end{equation*}

Let us denote this resolution by $J$. By Lemma
\ref{keineerweiterungen} it is sufficiently acyclic in the sense of
Proposition \ref{hinreichendAzyklisch}, and an argument analogous to
Lemma \ref{sufficientlyacyclic} shows that $\End J$ is
quasi-equivalent to $\End I$. The following proposition concludes
our proof.
\end{proof}

\begin{proposition}\label{formalityonepointstratification}
For J the resolution of $W \oplus C[1]$ introduced in the previous
proof, we have that the corresponding endomorphism algebra $\End J$
is formal.
\end{proposition}

\begin{proof}
We abbreviate $\End J$ to $\mc{E}$ and will use the notations from
the proof of Proposition \ref{formalitytrivialstratification} to
specify a basis of $\mc{E}$. Whenever there is need to carry the
$p$-degree of a morphism
\begin{equation*} (\dots 0, f , 0\dots) \in \prod\limits_{p \in \BZ}
\Hom(I^p, I^{p+m})
\end{equation*} along in
the notation, we will write $f^{(p)}$ instead of just $f$.

\begin{longtable}{|r|lcl|}
\hline m & basis of $\mc{E}^m$& & image under $d_\mc{E}$\\\hline\endhead\nopagebreak
\hline \endfoot \endlastfoot-1 & $p_1$ & $\mapsto$ & $e_{11} - e_{21}$\\
\hline 0 & $h_1$ & $\mapsto$ & $\he{11} + \he{12}$\\* & $h_2$ &
$\mapsto$ & $-\he{21} -\he{22}$\\* & $e_1$ & $\mapsto$ &
$(\he{21}-\he{11}) + (e_{11} - e_{12})$\\* & $e_2$ & $\mapsto$ &
$(\he{22}-\he{12}) + (e_{22} - e_{21})$\\* & $p_1^{(0)}$ & $\mapsto$
& 0\\* & $p_1^{(1)}$ & $\mapsto$ & $e_{21}-e_{11}$\\* & $p_2$ &
$\mapsto$ & $e_{12}-e_{22}$\\* & $e_{11}$ & $\mapsto$ & $\hp{21} -
\hp{11}$\\* & $e_{21}$ & $\mapsto$ & $\hp{21} - \hp{11}$\\
\hline 1 & $\he{11}$ & $\mapsto$ & $\hp{11} - \hp{12}$\\* &
$\he{12}$ & $\mapsto$ & $\hp{12} - \hp{11}$\\* & $\he{21}$ &
$\mapsto$ & $\hp{21} - \hp{22}$\\* & $\he{22}$ & $\mapsto$ &
$\hp{22} - \hp{21}$\\ & $\hp{11}$ & $\mapsto$ & $0$\\ & $\hp{21}$ &
$\mapsto$ & $0$\\ & $p_1$ & $\mapsto$ & $0$\\ & $e_{11}$ & $\mapsto$
& $\hp{11} - \hp{21}$\\ & $e_{12}$ & $\mapsto$ & $\hp{12} -
\hp{22}$\\* & $e_{21}$ & $\mapsto$ & $\hp{11} - \hp{21}$\\* &
$e_{22}$ & $\mapsto$ & $\hp{12} - \hp{22}$\\
\hline 2 & $\hp{11}$ & $\mapsto$ & $0$\\* & $\hp{12}$ & $\mapsto$ &
$0$\\* & $\hp{21}$ & $\mapsto$ & $0$\\* & $\hp{22}$ & $\mapsto$ &
$0$\\* \hline
\end{longtable}
For all other $m$ we have $\mc{E}^m = 0$. Next, we need to calculate
bases of the kernel and image of the differential $d_\mc{E}$:

\begin{longtable}{|l|l|}
\hline & basis\\
\hline $\ker d_\mc{E}^{-1}$ & $\emptyset$\\
$\im d_\mc{E}^{-1}$ & $\{e_{11} - e_{21}\}$\\
$\ker d_\mc{E}^0$ & $\{1, e_{11} - e_{21}, p_1^{(0)}\}$\\
$\im d_\mc{E}^0$ & $\{d_\mc{E}^0(h_1), d_\mc{E}^0(e_1),
d_\mc{E}^0(e_2), d_\mc{E}^0(p_1^{(1)}), d_\mc{E}^0(p_2),
d_\mc{E}^0(e_{11}) \}$\\
$\ker d_\mc{E}^1$ & $\{d_\mc{E}^0(h_1), d_\mc{E}^0(e_1),
d_\mc{E}^0(e_2), d_\mc{E}^0(p_1^{(1)}), d_\mc{E}^0(p_2),
\hp{11}, \hp{21}, p_1\}$\\
$\im d_\mc{E}^1$ & $\{d_\mc{E}^1(\he{11}), d_\mc{E}^1(\he{21}),
d_\mc{E}^1(e_{11})\}$\\
$\ker d_\mc{E}^2$ & $\{\hp{11}, \hp{12}, \hp{21}, \hp{22}\}$\\\hline
\end{longtable}
Thus the $-1^\text{st}$ cohomology of $\mc{E}$ vanishes. $H^0\mc{E}$
is free of rank 2, with basis $\{1, p_1^{(0)}\}$, $H^1\mc{E}$ is
free of rank 2 as well, with basis $\{\hp{11}, p_1\}$, while
$H^2\mc{E}$ is free of rank 1 with basis $\{\hp{11}\}$. Now consider
the morphism $H\mc{E}\ra \mc{E}$ that maps every homology class to
its representative that we just specified. Since our system of
representatives is closed under multiplication, this is a morphism
of dg algebras and is a quasi-isomorphism by construction.
\end{proof}

\subsection{The 2-sphere and $n$ points}
Finally, we would like to consider for $n \geq 2$ the stratification
$\Sn$, which decomposes the 2-sphere into $n$ points and their
complement. Since the big stratum is no longer simply connected, the
skyscrapers at the $n$ points and the constant sheaf generate only a
subcategory of the corresponding constructible derived category
$D^b_{c,\Sn}(\Kugelschale)$. However, since our focus is on
formality in general, we will show that the corresponding dg algebra
is formal as well. Our final result is the proof of Theorem
\ref{main2} from the introduction:

\begin{theorem}\label{npointstratification}
For $n \geq 2$, let $\mc{W}_1, \dots, \mc{W}_n$ be the skyscrapers
at the $n$ points of the stratification $\Sn$. Then for any
injective resolution $I$ of $\mc{W}_1\oplus \dots \oplus
\mc{W}_n\oplus \cs\Ring\Kugelschale[1]$, the dg algebra $\End I$ is
formal.
\end{theorem}

\begin{proof}
We consider the acyclic stratification $\An$ corresponding to $\Sn$,
which consists of $n$ points $P_1, \dots, P_n$ (for simplicity we
assume they are on the equator), the equator pieces in between,
$E_1, \dots, E_n$, and the two hemi-spheres $H_1, H_2$. The
corresponding quiver is given by
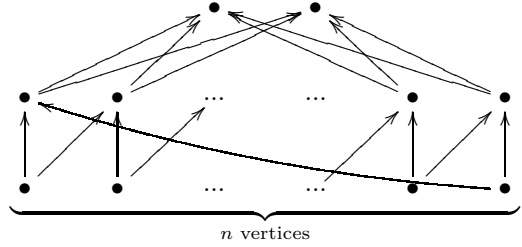
\begin{equation*} \underbrace{\xymatrix{&&\bullet & \bullet&&\\
\bullet \ar[urr] \ar[urrr]&\bullet \ar[ur] \ar[urr]&...&...& \bullet
\ar[ull] \ar[ul]&\bullet \ar[ulll] \ar[ull]\\ \bullet \ar[u] \ar[ur]
&\bullet \ar[u] \ar[ur] &...&...\ar[ur]& \bullet \ar[u] \ar[ur]&
\bullet \ar[u] \ar@/^/[ulllll]}}_{\text{$n$ vertices}}
\end{equation*}
where we identify by relations all paths having the same starting
and end point. Under the equivalence
$$\Rep(Q_{\An},\rho_{\An}) \sira
\Sh{\An} \sira \operatorname{Sh}_{w,\hspace{0,5mm}c,\hspace{0,5mm}
\An}(\Kugelschale)
$$
the constant sheaf $\cs\Ring\Kugelschale$ corresponds to the
representation $C$ that has $\Ring$ at every vertex and the identity
at every arrow. The skyscraper $\mc{W}_i$ corresponds to the
representation that has $\Ring$ at the $i$th vertex in the lower row
and zeros everywhere else. We denote this object by $W_i$.

Next, we need sufficiently acyclic objects. As before, we take for
every stratum $S$ of $\An$ the representation corresponding to the
constant sheaf along the closure of $S$. More precisely, let
$i:\ol{S}\hra\Kugelschale$ be the inclusion, and denote by $I_S$ the
image of $i_!\cs\Ring{\ol{S}}$ under the equivalence
$\operatorname{Sh}_{w,c,\An} \overset\sim\ra
\Rep(Q_{\An},\rho_{\An})$. This yields the following representation:
At the vertex corresponding to $S$ we have $\Ring$, as well as at
every vertex from which we have a path to the vertex $S$. The other
vertices get a zero. Finally we assign to an arrow the identity map,
if possible, and else the zero morphism. As in Lemma
\ref{keineerweiterungen} we see that those objects do not have
higher extensions.

Again, we need to study the morphism spaces between the objects
$I_S$. They vanish in all except for the following cases, in which
they are free of rank 1:
\begin{alignat*}{2} &\Hom(I_{H_j}, I_{H_j}) &
\quad\quad&\text{ for } j = 1 \text{ or } 2;\\
&\Hom(I_{H_j}, I_{E_i}) & \quad\quad& \text{ for } i = 1, \dots, n
\text{ und }j = 1,2;\\
&\Hom(I_{H_j}, I_{P_i}), &\quad\quad& \text{ for } i = 1, \dots, n
\text{ und
} j = 1,2;\\
&\Hom(I_{E_i}, I_{E_i}) & \quad\quad&\text{ for } i = 1, \dots,n;\\
&\Hom(I_{E_i}, I_{P_i}) & \quad\quad&\text{ for } i = 1, \dots, n;\\
&\Hom(I_{E_1}, I_{P_n}) \text{ and }\Hom(I_{E_i}, I_{P_{i-1}}) & \quad\quad&\text{ for } i = 2, \dots, n;\\
&\Hom(I_{P_i}, I_{P_i}) & \quad\quad&\text{ for } i = 1, \dots, n.
\end{alignat*}
Each of these morphism spaces is generated by a canonical morphism,
the one that is given by the identity on $\Ring$ wherever possible
and zeros everywhere else. We are going to use the same notations
for those morphisms as before: $h_j$, $e_i$ and $p_i$ are meant to
be the generators of $\Hom(I_{H_j}, I_{H_j})$,
$\Hom(I_{E_i},I_{E_i})$ and $\Hom(I_{P_i}, I_{P_i})$, respectively.
The generator of $\Hom(I_{H_j}, I_{E_i})$ is denoted by $\he{ji}$,
that of $\Hom(I_{H_j}, I_{P_i})$ by $\hp{ji}$ and that of
$\Hom(I_{E_i}, I_{P_k})$ by $e_{ik}$. For notational simplicity we
take indices modulo $n$, i.e. for $i=1$, we get $e_{i(i-1)}=e_{1n}$,
etc.

We now need to find an acyclic resolution of $W_1 \oplus\dots\oplus
W_n \oplus C$. With the above notations, we get the following
resolution:
\begin{align*}
&\hspace{5mm} I_{E_1} \oplus \ldots \oplus I_{E_n} \\
I_{H_1} \oplus I_{H_2}
\os{\scriptsize{\begin{pmatrix}\partial^0\\0\end{pmatrix}}}\ra
&\hspace{15mm}\oplus&\hspace{-10mm}\os{\scriptsize{\begin{pmatrix}
\partial^1&0\end{pmatrix}}}\ra
I_{P_1} \oplus \ldots \oplus I_{P_n}\\ \\&\hspace{5mm}I_{P_1} \oplus
\ldots \oplus I_{P_n}
\end{align*}
where $\partial^0$ is given by the matrix
$$\hspace{2mm}\begin{pmatrix}\he{11} & -\he{21}\\ \vdots & \vdots\\
\he{1n} & -\he{2n}\end{pmatrix}$$ and $\partial^1$ by
$$\begin{pmatrix}e_{11} &
-e_{21} &0& \cdots &0\\ 0 & e_{22}&0& \cdots&0\\
\vdots&0&\ddots& &\vdots\\0&\vdots&\cdots&0&-e_{n(n-1)}\\-e_{1n}
&0&\cdots&0&e_{nn}\end{pmatrix}$$

Using an argument similar to Lemma \ref{sufficientlyacyclic}, the
result follows by the next proposition.
\end{proof}

\begin{proposition}
For $J$ the resolution of $W_1 \oplus\dots\oplus W_n \oplus C$
introduced in the previous proof, the corresponding
endomorphism-dg-algebra $\End J$ is formal.
\end{proposition}

\begin{proof}
This proof is going to be a little more complex than those of the
analogous statements before. We are going to find a dg-sub-algebra
$\mc{U}$ of $\End J$ and a two-sided ideal $\mc{I}$ of $\mc{U}$,
that yield a sequence of quasi-isomorphisms
\begin{equation*}
\End J \hookleftarrow \mc{U} \tha \mc{U}/I \hookleftarrow
H(\mc{U}/\mc{I})
\end{equation*}

For brevity, we write $\mc{E}$ instead of $\End J$. Using the same
notations as before we would like to give a basis of the free module
$\mc{E}$:

\begin{longtable}{|r|l|lcll|}
\hline m & rank of $\mc{E}^m$ & basis of $\mc{E}^m$& & image under
$d_\mc{E}$ &\\\hline\endhead \hline \endfoot \endlastfoot -1 & $n$ &
$p_i$ & $\mapsto$ & $e_{ii} - e_{(i+1)i}$ & $i = 1, \dots, n$\\
\hline 0 & $5n+2$ & $h_1$ & $\mapsto$ & $\sum\limits_{i=1}^n
\he{1i}$& \\* &&$h_2$ & $\mapsto$ & $-\sum\limits_{i=1}^n \he{2i}$ &
\\* &&$e_i$ & $\mapsto$ & $(\he{2i}-\he{1i}) + (e_{ii}-e_{i(i-1)})$
& $i = 1,\dots, n$\\* &&$p_i^{(0)}$ & $\mapsto$ & 0 & $i = 1,\dots,
n$\\* &&$p_i^{(1)}$ & $\mapsto$ & $e_{(i+1)i} - e_{ii}$, &$i = 1,
\dots, n$\\* &&$e_{ii}$ & $\mapsto$ & $\hp{2i} - \hp{1i}$ & $i =
1,\dots, n$\\* &&$e_{(i+1)i}$ & $\mapsto$ & $\hp{2i} - \hp{1i}$ & $i
= 1,\dots, n$\\ \hline 1 & $7n$ & $\he{1i}$ & $\mapsto$ & $\hp{1i} -
\hp{1(i-1)}$ & $i = 1,\dots, n$\\ &&$\he{2i}$ & $\mapsto$ & $\hp{2i}
- \hp{2(i-1)}$ & $i = 1,\dots, n$\\ &&$\hp{ji}$ & $\mapsto$ & 0 &
$j=1,2$, $i = 1,\dots, n$\\ &&$e_{ii}$ & $\mapsto$ & $\hp{1i} -
\hp{2i}$ & $i = 1,\dots, n$\\ &&$e_{(i+1)i}$ & $\mapsto$ & $\hp{1i}
- \hp{2i}$ & $i = 1,\dots, n$\\* &&$p_i$ & $\mapsto$ & 0 & $i =
1,\dots, n$\\ \hline 2 & $2n$ & $\hp{ji}$ & $\mapsto$ & 0 &$j =
1,2$, $i = 1,\dots, n$\\ \hline
\end{longtable}

Next, we need to specify a basis of the kernel and image of the
differential $d_\mc{E}$:

\begin{longtable}{|l|l|}
\hline & basis\\
\hline $\ker d_\mc{E}^{-1}$ & $\emptyset$\\
$\im d_\mc{E}^{-1}$ & the basis of $\mc{E}^{-1}$ as in the table above\\
$\ker d_\mc{E}^0$ & $\{1, p_i^{(0)}, (e_{ii} - e_{(i+1)i})
\hspace{1mm}|\hspace{1mm}i=1, \dots, n\}$\\
$\im d_\mc{E}^0$ & $\{d_\mc{E}^0 (h_1), d_\mc{E}^0(e_i),
d_\mc{E}^0(p_i^{(1)}),
d_\mc{E}^0(e_{ii})\hspace{1mm}|\hspace{1mm}i=1, \dots, n\}$\\
$\ker d_\mc{E}^1$ & $\{d_\mc{E}^0(h_1), d_\mc{E}^0(e_i),
d_\mc{E}^0(p_i^{(1)}), \hp{ji}, p_i \hspace{1mm}|\hspace{1mm}j =
1,2, i=1, \dots, n\}$\\
$\im d_\mc{E}^1$ & $\{d_\mc{E}^1(\he{1i}), d_\mc{E}^1(\he{2i}),
d_\mc{E}^1(e_{11})\hspace{1mm}|\hspace{1mm} i=2, \dots, n\}$\\
$\ker d_\mc{E}^2$ & the basis of $\mc{E}^2$ as in the table above\\
\hline
\end{longtable}

From this table we can read the cohomology of $\mc{E}$:
\begin{align*}
&H^{-1}\mc{E} = 0, \hspace{4mm} H^0\mc{E} = \langle 1, \hspace{2mm}
p_i^{{(0)}} \rangle_{i=1, \dots, n}, \hspace{4mm} H^1\mc{E} =
\langle \hp{1i}, p_i \rangle_{i=1, \dots, n} &\text{and } H^2\mc{E}
= \langle \hp{11} \rangle.
\end{align*}

Unfortunately, unlike before, we cannot find a quasi-isomorphism
$H\mc{E} \ra \mc{E}$. This is because in this situation, there is no
system of representants of the generators of $H\mc{E}$ that form a
sub-algebra of $\mc{E}$. The problem is that when multiplying the
representants of $H^1\mc{E}$, we get all of the $\hp{1i}$. Hence the
first step is to identify the $\hp{1i}$, which can be done by
defining a subalgebra $\mc{U}$ of $\mc{E}$ and dividing it by a
suitable ideal $\mc{I}$.

Consider the dg algbra $\mc{U}$ given by the following basis:

\begin{longtable}{|r|l|lcll|}
\hline m & rank of $\mc{U}^m$ & basis of $\mc{U}^m$& & image under
$d_\mc{U}$ &\endhead
\hline -1 & 0 & $\emptyset$ &&&\\
\hline 0 & $2n + 3$ & $h_1$ & $\mapsto$ & $\sum\limits_{i=1}^n
\he{1i}$& \\* &&$h_2$ & $\mapsto$ & $-\sum\limits_{i=1}^n \he{2i}$ &
\\* &&$e_i$ & $\mapsto$ & $(\he{2i}-\he{1i}) + (e_{ii}-e_{i(i-1)})$
& $i = 1,\dots, n$\\* &&$p_i^{(0)}$ & $\mapsto$ & 0 & $i = 1,\dots,
n$\\* &&$\sum_{i=1}^n p_i^{(1)}$ & $\mapsto$ & $\sum_{i=1}^n
(e_{(i+1)i} -
e_{ii})$ &\\
\hline 1 & $5n+1$ & $\he{1i}$ & $\mapsto$ & $\hp{1i} - \hp{1(i-1)}$
& $i = 1,\dots, n$\\* && $\he{2i}$ & $\mapsto$ & $\hp{2i} -
\hp{2(i-1)}$ & $i = 1,\dots, n$\\* &&$\hp{1i}$ & $\mapsto$ & 0 & $i
= 1,\dots, n$\\* &&$e_{11}$ & $\mapsto$ & $\hp{11} - \hp{21}$ &\\*
&&$e_{1n}$ & $\mapsto$ & $\hp{1n} - \hp{2n}$ & \\* &&$e_{ii} -
e_{i(i-1)}$ & $\mapsto$ & $\hp{1i} - \hp{2i} - (\hp{1(i-1)} -
\hp{2(i-1)})$ & $i = 2, \dots, n$\\*
&&$p_i$ & $\mapsto$ & 0 & $i = 1,\dots, n$\\
\hline 2 & $2n$ & $\hp{ji}$ & $\mapsto$ & 0 & $j = 1,2$, $i =
1,\dots, n$\\ \hline
\end{longtable}
\noindent It is easy to see that $\mc{U}$ is closed by
multiplication, and that the inclusion $\mc{U} \hra \mc{E}$ is a
quasi-isomorphism.

Consider the two-sided ideal $\mc{I}$ of $\mc{U}$ that is given by
$\mc{I}^0 = 0$, $\mc{I}^1 = \langle \he{1i} \rangle_{i = 2, \dots,
n}$ and $\mc{I}^2 = \langle \hp{1i} - \hp{1(i-1)} \rangle_{i = 2,
\dots, n}$. The cohomology of $\mc{I}$ vanishes, hence by the long
exact cohomology sequence the projection $\mc{U} \tha \mc{U}/\mc{I}$
is a quasi-isomorphism.

Now in $\mc{U}/\mc{I}$ we have the following relations:
$$\hp{11} = \hp{1i} \text{ for } i = 2, \dots, n$$
Hence the following system of generators of the cohomology of
$\mc{U}/\mc{I}$ is closed under multiplication:
$$H^0(\mc{U}/\mc{I}) = \langle 1,
p_i^{{(0)}} \rangle_{i=1, \dots, n}, \hspace{1mm} H^1(\mc{U}/\mc{I})
= \langle \hp{1i}, p_i \rangle_{i=1, \dots, n} \text{ und }
H^2(\mc{U}/\mc{I}) = \langle \hp{11} \rangle$$ Mapping each of those
generators to the corresponding representant yields the final
quasi-isomorphism in the chain
$$\mc{E} \hookleftarrow \mc{U} \tha \mc{U}/I \hookleftarrow
H(\mc{U}/\mc{I})$$ and we are done.
\end{proof}

\bibliographystyle{amsalpha}

\end{document}